\documentclass[
final
]{dmtcs_episciences}

\usepackage{xcolor}
\usepackage{cite}
\usepackage{enumerate}
\usepackage{amsthm,amsmath}
\usepackage{mathrsfs,changes}
\usepackage{a4,latexsym,epsfig,amssymb,amsfonts,enumerate,tikz}
\usetikzlibrary{matrix}
\usetikzlibrary{decorations.pathreplacing}
\tikzstyle{overbrace text style}=[above, pos=.5, yshift=3mm]
\tikzstyle{overbrace style}=[decorate,decoration={brace,raise=2mm ,amplitude=3pt}]

\newtheorem{theorem}{Theorem}[section]
\newtheorem{definition}{Definition}[section]
\newtheorem{lemma}[theorem]{Lemma}
\newtheorem{corollary}[theorem]{Corollary}
\newtheorem{open problem}{Open Problem}

\newtheorem{observation}{Observation}

\newcommand{\G}{\mbox{$\cal G$}}
\newcommand{\T}{\mbox{$\cal T$}}
\newcommand{\cp}{\mbox{$\cal P$}}
\newcommand{\B}{\mbox{$\cal B$}}

\title{Asymptotically sharpening the $s$-Hamiltonian index bound}
\author{Sulin Song\affiliationmark{1}
\and Lan Lei\affiliationmark{2}\thanks{This project is supported in part  by General Project of Natural Science Foundation of Chongqing, China (No. cstc2019jcyj-msxmX0579).}
\and Yehong Shao\affiliationmark{3}
\and Hong-Jian Lai\affiliationmark{1}\thanks{This project is supported in part by Natural Science Foundation of China grant (Nos.11771039, 11771443).}
}

\affiliation{
  Department of Mathematics, West Virginia University, Morgantown, WV 26506, USA.\\
  School of Mathematics and Statistics, Chongqing Technology and Business University, Chongqing 400067, P.R. China.\\
  Arts and Science, Ohio University Southern, Ironton, OH 45638, USA.}
  
\keywords{$s$-Hamiltonian; $(s,t)$-supereulerian; Collapsible graphs; $k$-triangular;  Line graph stable properties}  

\received{2021-09-15}
\revised{2022-05-23}
\accepted{2022-05-24}

\begin{document}
\publicationdetails{24}{2022}{1}{22}{8484}
\maketitle

\begin{abstract}
For a non-negative integer $s\le |V(G)|-3$, a graph $G$ is $s$-Hamiltonian if the removal of any $k\le s$ vertices results in a Hamiltonian graph.
Given a connected simple graph $G$ that is not isomorphic to a path, a cycle, or a $K_{1,3}$, let $\delta(G)$ denote the minimum degree of $G$,
let $h_s(G)$ denote the smallest integer $i$ such that
the iterated line graph $L^{i}(G)$ is $s$-Hamiltonian, and let $\ell(G)$ denote the length of the longest non-closed path $P$ in which all internal vertices have degree 2 such that $P$ is not both of length 2 and in a $K_3$.
For a simple graph $G$, we establish better upper bounds for $h_s(G)$ as follows.
\begin{equation*}
h_s(G)\le \left\{
\begin{aligned}
& \ell(G)+1, &&\mbox{ if }\delta(G)\le 2 \mbox{ and }s=0;\\
& \widetilde d(G)+2+\lceil \lg (s+1)\rceil, &&\mbox{ if }\delta(G)\le 2  \mbox{ and }s\ge 1;\\
& 2+\left\lceil\lg\frac{s+1}{\delta(G)-2}\right\rceil, && \mbox{ if } 3\le\delta(G)\le s+2;\\
& 2, &&{\rm otherwise},
\end{aligned}
\right.
\end{equation*}
where $\widetilde d(G)$ is the smallest integer $i$ such that $\delta(L^i(G))\ge 3$.
Consequently, when $s \ge 6$, this new upper bound for the $s$-hamiltonian index
implies that $h_s(G) = o(\ell(G)+s+1)$
as $s \to \infty$. This sharpens the result, $h_s(G)\le\ell(G)+s+1$, obtained by Zhang et al. in [Discrete Math., 308 (2008) 4779-4785].
\end{abstract}


\section{Introduction}

Finite loopless graphs permitting parallel edges are considered with undefined terms being referenced to \cite{BoMu08}.
As in \cite{BoMu08}, a simple graph is one that is loopless and without parallel edges, and the minimum degree of a graph $G$ is denoted by $\delta(G)$.
For a subset $X \subseteq V(G)$ or $E(G)$, let $G[X]$ denote the subgraph induced by $X$,
and let $G - X = G[V(G) - X]$ or $G[E(G) - X]$, respectively.
When $X=\{x\}$, we write $G-x$ for $G-\{x\}$.
Throughout this paper, if $X\subseteq E(G)$, then, for notational convenience, we often use $X$ to denote both the edge subset of $E(G)$ and $G[X]$.
We also use $\lg x$ as an alternative notation for $\log_2 x$,
and set $[m,n]=\{m,m+1,\ldots, n\}$ for two integers $m,n$ with $m\le n$.

A graph is considered \emph{Hamiltonian} if it has a spanning cycle.
For a non-negative integer $s\le |V(G)|-3$, a graph is called \emph{$s$-Hamiltonian} if the removal of any $k\le s$ vertices results in a Hamiltonian graph.
A subgraph $H$ of $G$ is \emph{dominating} if $G-V(H)$ is edgeless.
Following \cite{BoST77, Catl92}, a graph is \emph{supereulerian} if it has a spanning closed trail.
Harary and Nash-Williams \cite{HaNa65} characterized Hamiltonian line graphs as follows, which implies that the line graph of every supereulerian graph is Hamiltonian.

\begin{theorem}[Harary and Nash-Williams, Proposition 8 of \cite{HaNa65}]
\label{hanathm}
Let $G$ be a graph with at least three edges. Then
$L(G)$ is Hamiltonian if and only if $G$ has a dominating closed trail.
\end{theorem}

The \emph{line graph} of a graph $G$, denoted $L(G)$, is a simple graph with $E(G)$ being its
vertex set, where two vertices in $L(G)$ are adjacent whenever the corresponding edges in $G$ are adjacent.
A \emph{claw-free} graph is one that does not have an induced subgraph isomorphic to $K_{1,3}$.
Beineke \cite{Bein68} and Robertson (Page 74 of \cite{Ha69})
showed that line graphs are claw-free graphs.
For a positive integer $i$, we define $L^0(G)=G$, and the $i$th iterated line
graph of $G$, denoted $L^i(G)$, is defined recursively as $L^{i}(G) =
L(L^{i-1}(G))$.

Let $J_1$ and $J_2$ be two graphs obtained from $K_{1,3}$ via identifying two and three vertices of degree 1, respectively. Let $K_{1,3}^+=\{J_1,J_2,K_{1,3}\}$.
Since the line graph of a cycle remains unchanged, in general, we assume that graphs
are not isomorphic to paths, cycles or any members in $K_{1,3}^+$. For this reason, we define
\[\G = \{G: G\mbox{ is connected and is not isomorphic to a path, or a cycle, or a member in }K^+_{1,3}\}.\]
Chartrand in \cite{Char68} introduced and studied the Hamiltonian index of a graph,
and initiated the study of indices of graphical properties.
More generally, we have the following definition.

\begin{definition}[Definition 5.8 of \cite{LaSh07}]
For a property $\mathcal{P}$, the \emph{$\mathcal{P}$-index} of $G\in\G$ is defined by
\begin{equation*}
 \mathcal{P}(G)=\left\{
\begin{aligned}
&\min\{i: L^i(G)\ {\rm has\ property}\ \mathcal{P}\},\ && {\rm if\ one\ such\ integer\ }i\ {\rm exists};\\
&\infty, &&{\rm otherwise}.
\end{aligned}
\right.
\end{equation*}
\end{definition}

A property $\cp$ is {\bf line graph stable} if $L(G)$ has $\cp$ whenever $G$ has $\cp$.
Chartrand \cite{Char68} showed that for every graph $G\in\G$, the Hamiltonian index exists as a finite number, and
the characterization of Hamiltonian line graphs (Theorem \ref{hanathm}) by Harary and Nash-Williams implies that being Hamiltonian is line graph stable.
Z. Ryj\'{a}\v{c}ek et al. \cite{RyWX11} indicated that determining the value of the Hamiltonian index is difficult.
Clark and Wormald \cite{ClWo83} showed that for all graphs in $\G$, other Hamiltonian-like indices also exist as finite numbers;
and in \cite{LaSh07}, it is shown that these Hamiltonian-like properties are also line graph stable.
Let $h(G), h_s(G)$ and $s(G)$ be the \emph{Hamiltonian index}, \emph{$s$-Hamiltonian index} and \emph{supereulerian index} of $G\in \G$, respectively. By definitions, $h(G)=h_0(G)$.

Let $P=v_0e_1v_1e_2\cdots v_{s-1}e_{s}v_{s}$ be a path of a graph $G$ where each $e_i\in E(G)$ and each $v_i\in V(G)$. Then $P$ is called a $(v_0,v_l)$-path or an $(e_1,e_s)$-path of $G$.
A path $P$ of $G$ is \emph{divalent} if every internal vertex of $P$ has degree 2 in $G$.
For two non-negative integers $s$ and $t$, a divalent path $P$ of $G$ is a
\emph{divalent $(s,t)$-path} if the two end vertices of $P$ have degrees $s$ and $t$, respectively.
A non-closed divalent path $P$ is considered \emph{proper} if $P$ is not both of length 2 and in a $K_{3}$.
As in \cite{LaiH88, ZELS08}, for a graph $G \in \G$, define
\begin{equation} \label{def-ell}
\ell(G)=\max\{m: G\ \mbox{has a length $m$ proper divalent path}\}.
\end{equation}
\begin{theorem}\label{knownthm}
Let $G\in \G$ be a simple graph. Each of the following holds.\\
(i) {\rm(Corollary 6 of \cite{LaiH88})} $s(G)\le \ell(G)$.\\
(ii) {\rm(Corollary 6 of \cite{LaiH88})} $h(G)\le s(G)+1\le \ell(G)+1$.\\
(iii) {\rm(Zhang et al., Theorem 1.1 of \cite{ZELS08})} $h_s(G)\le\ell(G)+s+1$.
\end{theorem}

Several natural questions arise here. Can we improve the upper bounds above? Can we generalize Theorem \ref{knownthm}(i) in the way as Theorem \ref{knownthm}(iii) extends
Theorem \ref{knownthm}(ii)?
As a generalization of supereulerian graphs, given two non-negative integers $s$ and $t$,
it is defined in \cite{LLWL10} that a graph $G$ is \emph{$(s,t)$-supereulerian} if for any disjoint
sets $X,Y\subset E(G)$ with $|X|\leq s$ and $|Y|\leq t$, $G - Y$ contains a spanning closed trail that traverses all edges in $X$.
Former studies on $(s,t)$-supereulerian graphs can be found in \cite{LeLW07, LeLi08, LLWL10}, among others.
Let $i_{s,t}(G)$ denote the \emph{$(s,t)$-supereulerian index} of a graph $G\in\G$. Thus, $i_{0,0}(G)=s(G)$.
By the characterization of Hamiltonian line graphs (Theorem \ref{hanathm}),
the line graph of every $(0,s)$-supereulerian graph is $s$-Hamiltonian, and then we obtain the following observation.

\begin{observation}\label{obs1}
Let $G\in\G$. Then $h_s(G)\le i_{0,s}+1$. In particular, $h(G)\le s(G)+1$.
\end{observation}

To present the main results, an additional notation would be needed.
Since $G\in\G$, it is observed that (for example, Theorem 18 of \cite{CLXY09})
there exists an integer $i>0$ such that $\delta(L^i(G))\geq 3$.
Define
\begin{equation} \label{d3}
\widetilde d(G)=\min\{i:\delta(L^i(G))\geq 3\}.
\end{equation}
Our main results can now be stated as follows.
\begin{theorem}\label{mainthm}
Let $G\in \G$ be a simple graph with $\delta=\delta(G)$ and $\widetilde d=\widetilde d(G)$. Then, given two non-negative integers $s$ and $t$,
\begin{equation} \label{ist}
i_{s,t}(G)\le \left\{
\begin{aligned}
& \ell(G), &&\mbox{ if }\delta\le 2 \mbox{ and }s=t=0;\\
& \widetilde d+1+\lceil \lg (s+t+1)\rceil, &&\mbox{ if }\delta\le 2  \mbox{ and }s+t\ge 1;\\
&1+\left\lceil\lg\frac{s+t+1}{\delta-2}\right\rceil, && \mbox{ if } 3\le\delta\le s+t+2;\\
& 1, &&{\rm otherwise}.
\end{aligned}
\right.
\end{equation}
\end{theorem}

Using Observation \ref{obs1}, Theorem \ref{mainthm} implies Corollary \ref{maincor} below.

\begin{corollary}\label{maincor}
Let $G\in \G$ be a simple graph with $\delta=\delta(G)$ and $\widetilde d=\widetilde d(G)$. Then, given a non-negative integer $s\le |V(G)|-3$,
\begin{equation}
h_s(G)\le \left\{
\begin{aligned}
& \ell(G)+1, &&\mbox{ if }\delta\le 2 \mbox{ and }s=0;\\
& \widetilde d+2+\lceil \lg (s+1)\rceil, &&\mbox{ if }\delta\le 2  \mbox{ and }s\ge 1;\\
& 2+\left\lceil\lg\frac{s+1}{\delta-2}\right\rceil, && \mbox{ if } 3\le\delta\le s+2;\\
& 2, &&{\rm otherwise}.
\end{aligned}
\right.
\end{equation}
\end{corollary}

Given a simple graph $G\in \G$ with $\ell=\ell(G)$ and $\widetilde d=\widetilde d(G)$.
By the formula to compute $\widetilde d$ to be presented in Section \ref{sect_d}, we have 
$\widetilde d\le\ell+2$.
When $s\ge 6$, as $\lceil\lg(s+1)\rceil+2\le s-1$, we have
$\widetilde d+2+\lceil\lg(s+1)\rceil\le\ell+1+s$. Moreover, since $\lceil\lg(s+1)\rceil=o(s)$ as $s\rightarrow \infty$, it follows that $\widetilde d+2+\lceil\lg(s+1)\rceil=o(\ell+s+1)$ as $s\rightarrow \infty$.
Similarly, when $s\ge 1$ and $n\ge 1$, we have $\left\lceil\lg\frac{s+1}{n}\right\rceil\le s$ and $\left\lceil\lg\frac{s+1}{n}\right\rceil=o(s)$ as $s\rightarrow \infty$.
It means that $2+\left\lceil\lg\frac{s+1}{n}\right\rceil\le s+2$ and $2+\left\lceil\lg\frac{s+1}{n}\right\rceil=o(s+2)$ as $s\rightarrow \infty$.
Hence, when $s \ge 6$,
the upper bounds above sharpen the result of Theorem \ref{knownthm}(iii).

In the next section, we present preliminaries and tools that will be used in our discussions.
In Section 3, we shall show some important lemmas, including a corrected formula to compute $\widetilde d(G)$, which are very helpful to prove the main result, Theorem \ref{mainthm}, in Section 4.

\section{Preliminaries}

For a vertex $v\in V(G)$, we denote $N_G(v)$ to be the set of all neighbors of vertex $v$ in a graph $G$, and denote $E_G(v)$ to be the set of all edges incident with $v$ in $G$.
The degree of vertex $v$ is denoted by $d_G(v)=|E_G(v)|$.
For integer $i\ge 0$, let $D_i(G)$ be the set of all vertices of degree $i$ in $G$, and let $O(G)$  be the set of all odd degree vertices in $G$.
For an edge $e\in E(G)$, let $E_G(e)$ be the set of all edges adjacent with $e$ in $G$ and $d_G(e)=|E_G(e)|$.
For notational convenience, if $v\in V(G)$ and $e\in E(G)$, we write $d(v)$ and $d(e)$ for $d_G(v)$ and $d_G(e)=d_{L(G)}(e)$, respectively, when $G$ or $L(G)$ is understood from context.

\subsection{Iterated Line Graphs}

For a subset $X\subseteq E(G)$, let $L^0(X)=X$ and $L^i(X)=L^i(G)[L^{i-1}(X)]$ for each integer $i\ge 1$.
Moreover, for a subset $Y\subseteq E(L^i(G))$, by the definition of the iterated line graphs, there exists a unique $Z\subseteq E(L^{i-j}(G))$ for each $j\in[0,i]$ such that $L^{j}(Z)=Y$, denoted $Z=L^{-j}(Y)$.
Thus, for two integers $i,j$ and an edge subset $X\subseteq E(G)$, $L^{i}(L^j(X))=L^{i+j}(X)$.

\begin{lemma}\label{(s,t)path}
Given an integer $i\geq 0$ and a graph $G$.
If $P$ is a divalent $(s,t)$-path in $L^i(G)$ of length $r$ that is not in a $K_3$, then for each $j\in[0,i]$,
$L^{-j}(P)$ is a divalent $(s,t)$-path in $L^{i-j}(G)$ of length $r+j$.
\end{lemma}

\begin{proof}
Assume that $j_0$ is the smallest number such that $L^{-j_0}(P)$ is not a divalent $(s,t)$-path of length $r+j_0$ where $0< j_0\le i$.
Let $Q=L^{-j_0+1}(P)$. Thus, $Q$ is a divalent $(s,t)$-path in $L^{i-(j_0-1)}(G)$ of length $r+j_0-1$.
First, we claim that $Q$ is not in a $K_3$. If $Q$ is in a $K_3$, then
$P=L^{j_0-1}(Q)$ is in a $K_3$ since the line graph of a $K_3$ is still a $K_3$,
which contradicts the assumption of $P$ being not in a $K_3$.

Now, set $J=L^{i-j_0}(G)$, and then $L(J)=L^{i-(j_0-1)}(G)$.
Let $Q$ be a $(u,v)$-path of $L(J)$, where $u\in D_s(L(J))$ and $v\in D_t(L(J))$.
As $Q$ is not in a $K_3$ and the definition of divalent paths,
$L^{-j_0}(P)=L^{-1}(Q)$ is a divalent $(u,v)$-path in $J$, where $\{u,v\}\subset E(J)$.
Let $L^{-j_0}(P)$ be a $(x,y)$-path where $\{x,y\}\subset V(J)$.
Since $d(x)=d(u)-2+2=s$ and $d(y)=d(v)-2+2=t$,
$L^{-j_0}(P)$ is a divalent $(s,t)$-path of length $r+j_0$, which contradicts our choice of $j_0$.
\end{proof}

\subsection{Collapsible Graphs}

In \cite{Catl88}, Catlin defined collapsible graphs as a useful tool to study supereulerian graphs.
A graph $G$ is \emph{collapsible} if for every subset $R\subseteq V(G)$ with $|R| \equiv 0$ (mod 2),
$G$ has a subgraph $\Gamma_R$ such that $O(\Gamma_R)=R$ and $G-E(\Gamma_R)$ is connected.
By definition, all complete graphs $K_n$ except $K_2$ are collapsible.
As shown in Proposition 1 of \cite{LaSY13},
a graph $G$ is collapsible if and only if for every subset $R\subseteq V(G)$ with $|R|\equiv 0$ (mod 2),
$G$ has a spanning connected subgraph $L_R$ with $O(L_R)=R$. As
$L_{\emptyset}$ is a spanning eulerian subgraph, every collapsible graph is supereulerian.
Collapsible graphs have been considered to be very useful to study eulerian subgraphs via the graph contraction.
For an edge subset $X\subseteq E(G)$, the \emph{contraction} $G/X$ is obtained from $G$
by identifying the two ends of each edge in $X$ and deleting the resulting loops.
If $H$ is a subgraph of $G$, then we write $G/H$ for $G/E(H)$.
The following theorem summarizes some useful properties of collapsible graphs for our proofs.

\begin{theorem}\label{collapsible} Let $G$ be a graph and $H$ be a subgraph of $G$.
\\
(i) {\rm (Catlin, Theorem 3 of \cite{CIJS90})}
If each edge of a connected graph $G$ is in a cycle of length $2$ or $3$,
then $G$ is collapsible.
\\
(ii) {\rm (Catlin, Corollary of Theorem 3 of \cite{Catl88})} If $H$ is collapsible,
then $G$ is collapsible if and only if $G/H$ is collapsible.
\end{theorem}

\section{The $k$-Triangular Index}

A cycle of length 3 is often called a triangle.
Following \cite{BrVe87}, for an integer $k> 0$, a graph $G$ is \emph{$k$-triangular}
if every edge lies in at least $k$ distinct triangles in $G$; a graph $G$ is \emph{triangular}
if $G$ is $1$-triangular.
Let $\T_k$ denote the family of all $k$-triangular graphs.
Thus, $\delta(G)\ge k+1$ if $G\in\T_k$.

Triangular graphs are often considered as models for some kinds of cellular networks (\cite{Have01})
and for certain social networks (\cite{LeRU14}), as well as mechanisms to study network stabilities and
to classify spam websites (\cite{BBCA08}).
In addition to its applications in the hamiltonicity of line graphs (\cite{BrVe87}),
triangular graphs are also related to design theory.
In 1984, Moon in \cite{Mo84} introduced the Johnson graphs $J(n,s)$, named after Selmer M. Johnson
for the closely related Johnson scheme.
The vertex set of  $J(n,s)$ is all $s$-element subsets of an $n$-element set,
where two vertices are adjacent
whenever the intersection of the corresponding two subsets contains exactly $s-1$ elements.
For example, $J(n,1)$ is isomorphic to $K_n$.
By definitions, for any integers $n\ge 3$ and $s$ with  $n > s$,
$J(n,s)$ is $(n-2)$-triangular. Therefore, it is of interests to investigate
$k$-triangular graphs for a generic value of $k$.

For an integer $k>0$, define $t_k(G)$ to be the \emph{$k$-triangular index} of $G \in \G$, that is,
the smallest integer $m$ such that $L^m(G)\in\T_k$.
The triangular index $t_1(G)$ is first investigated by Zhang et al.

\begin{theorem} \label{old1}
Let $G\in\G$ be a simple graph. Each of the following holds.
\\
(i) {\rm (Zhang et al., Proposition 2.3 (i) of \cite{ZSCXZ12})} Being triangular is line graph stable.
\\
(ii) {\rm (Zhang et al., Lemma 3.2 (iii) of \cite{ZELS08})} $t_1(G) \le \ell(G)$.
\end{theorem}

One of the purposes of this section is to determine, for any
positive integer $k$, the best possible bounds for
$t_k(G)$ and to investigate whether being $k$-triangular is line graph stable.

\subsection{A Formula to Compute $\widetilde d(G)$\label{sect_d}}

Recall that $\widetilde d(G) = \min\{i:\delta(L^i(G))\geq 3\}$, which is defined in (\ref{d3}).
Define
\begin{equation}\label{ell-i}
\begin{aligned}
& \ell_1(G) = \max\{|E(P)|: P \mbox{ is a divalent $(1,3)$-path of $G$}\},
\\
& \ell_2(G) = \max\{|E(P)|: P \mbox{ is a divalent $(1,t)$-path of $G$, where $t \ge 4$}\},
\\
& \ell_3(G) = \max\{|E(P)|: P \mbox{ is a divalent $(s,t)$-path of $G$, where $s, t\ge 3$}\},
\end{aligned}
\end{equation}
and \[\ell_0(G)=\max\{\ell_1(G)+1,\ell_2(G),\ell_3(G)-1\}.\]
In  \cite{KnNi03}, it is claimed that ``It is easy to see
$\widetilde d(G)=\ell_0(G)$."
However, there exists an infinite family of graphs each of which shows that this claim might be incorrect.
Let $\B = \{T: T$ is a tree with $V(T) = D_1(T) \cup D_3(T)\}$. 
For each $G \in \B$, we have $\ell_1(G)=\ell_3(G)=1$ and $\ell_2(G)=0$. Direct computation indicates
that $\widetilde d(G)=3 >\ell_0(G)$. See Figure \ref{fig1} for an illustration.

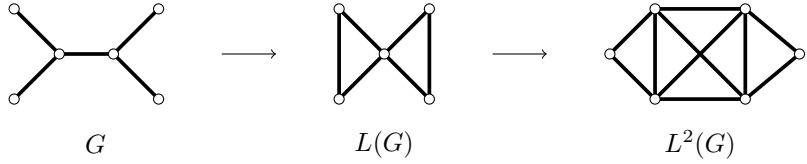
\begin{figure}[th]
\centering
\begin{tikzpicture}
[point/.style = {draw, circle, fill=white, inner sep = 0.3ex},scale=.6]

\draw[double={black}](0,0)--(1,1);
\draw[double={black}](0,2)--(1,1);
\draw[double={black}](1,1)--(2.2,1);
\draw[double={black}](2.2,1)--(3.2,0);
\draw[double={black}](2.2,1)--(3.2,2);

\draw[->] (4.6,1) -- (5.8,1);

\node (origin) at (0,0) [point]{};
\node (origin) at (0,2) [point]{};
\node (origin) at (1,1) [point]{};
\node (origin) at (2.2,1) [point]{};
\node (origin) at (3.2,0) [point]{};
\node (origin) at (3.2,2) [point]{};
\node at (1.8,-1) {$G$};

\draw[double={black}](7.2,0)--(8.2,1);
\draw[double={black}](7.2,2)--(8.2,1);
\draw[double={black}](7.2,0)--(7.2,2);
\draw[double={black}](8.2,1)--(9.2,0);
\draw[double={black}](8.2,1)--(9.2,2);
\draw[double={black}](9.2,2)--(9.2,0);

\draw[->] (10.6,1) -- (11.8,1);

\node (origin) at (7.2,0) [point]{};
\node (origin) at (7.2,2) [point]{};
\node (origin) at (8.2,1) [point]{};
\node (origin) at (9.2,0) [point]{};
\node (origin) at (9.2,2) [point]{};
\node at (8.2,-1) {$L(G)$};

\draw[double={black}](13.2,1)--(14.2,0);
\draw[double={black}](13.2,1)--(14.2,2);
\draw[double={black}](14.2,2)--(14.2,0);
\draw[double={black}](16.2,2)--(14.2,2);
\draw[double={black}](16.2,0)--(14.2,2);
\draw[double={black}](16.2,0)--(14.2,0);
\draw[double={black}](16.2,2)--(14.2,0);
\draw[double={black}](16.2,0)--(16.2,2);
\draw[double={black}](16.2,0)--(17.4,1);
\draw[double={black}](16.2,2)--(17.4,1);

\node (origin) at (13.2,1) [point]{};
\node (origin) at (14.2,0) [point]{};
\node (origin) at (14.2,2) [point]{};
\node (origin) at (16.2,0) [point]{};
\node (origin) at (16.2,2) [point]{};
\node (origin) at (17.4,1) [point]{};
\node at (15.2,-1) {$L^2(G)$};

\end{tikzpicture}
\caption{A member $G \in \B$ and its iterated line graphs.}
\label{fig1}
\end{figure}

Thus what would be the correct formula to compute $\widetilde d(G)$ becomes a question to be answered.
Before presenting our answer to it, we need some notation.
Let $U=\{v\in V(G): |N_G(v)|=1\}$ and $F=\bigcup_{v\in U}E_G(v)$.

\begin{lemma}\label{mg}
Let $G\in\G$ be a graph with $\delta(G)\leq 2$, $\widetilde d=\widetilde d(G)$ and $\ell_0=\ell_0(G)$.
The formula below computes $\widetilde d$:
\begin{equation}\label{e_m}
\widetilde d=
\left\{
\begin{array}{ll}
\max\{\ell_0,3\}, & \mbox{ if $|E_G(v)\cap F|=2$ for some $v\in D_3(G)$};
\\
\ell_0, & \mbox{ otherwise.}
\end{array} \right.
\end{equation}
\end{lemma}

\begin{proof}
Let $m$ be the right-hand side of (\ref{e_m}) and let $\ell_i = \ell_i (G)$ for each $i\in \{1, 2, 3\}$.
Then $m\le\widetilde d$ by definitions of $\widetilde d$ and line graphs. Now, it suffices to show that $\delta(L^m(G))\geq 3$. We assume that
$\delta(L^m(G))\leq 2$ to seek a contradiction.

If $\delta(L^m(G))=1$, then $L^m(G)$ has a divalent $(1,t)$-path of length $r$ where $t\geq 3$. By Lemma \ref{(s,t)path},
$G$ has a divalent $(1,t)$-path of length $r+m$.
If $t=3$, then $m+1\leq m+r\leq \ell_1\leq m-1$, a contradiction;
if $t>3$, then $m+1\leq m+r\leq \ell_2\leq m$, which is also a contradiction.

Then, $\delta(L^m(G))=2$. Pick $u\in D_2(L^m(G))$.
If $u$ is not in any triangles of $L^m(G)$, then $u$ is in a divalent $(s',t')$-path of length $r'\geq 2$ in $L^m(G)$ that is not in a $K_3$, where $s'\ge 3$ and $t'\ge 3$.
It follows that $G$ has a divalent $(s',t')$-path of length $r'+m$ by Lemma \ref{(s,t)path}, which shows that $2+m\leq r'+m\leq \ell_3\leq m+1$, a contradiction.
Thus, $u\in V(H)$ where $H\cong K_3$ is a subgraph of $L^m(G)$.
By the definition of line graphs, $L^{-1}(H)$ is isomorphic to one member of $\{K_3, K_{1,3}, J_1, J_2\}$.
Let $u=xy\in E(L^{-1}(H))$.

When $L^{-1}(H)\cong K_{1,3}$, as $d(u)=2$, we have $\ell_1(L^{m-1}(G))\geq 1$. By Lemma \ref{(s,t)path}, $\ell_1\geq 1+(m-1)=m\ge \ell_1+1$, a contradiction.

When $L^{-1}(H)\cong J_1$ or $J_2$, as there is no parallel edges in line graphs, $m=1$.
If $L^{-1}(H)\cong J_2$, then $G\cong J_2$ as $d(u)=2$, contradicting the definition of $\G$.
Then, $L^{-1}(H)\cong J_1$. If $u=xy$ is one of the parallel edges of $J_1$, then one of end vertices of $u$, say $x$, of degree 3 in $G$ satisfies $|E_G(x)\cap F|=2$, which implies $m\ge 3$ by (\ref{e_m}).
It is a contradiction with $m=1$.

When $L^{-1}(H)\cong K_3$, we have $d(x)=d(y)=2$ and $\ell_3\geq 3$ as $d(u)=2$.
If $m=1$, as $\ell_3\geq 3$, then $1=m\geq \ell_3-1\geq 2$, a contradiction. So, $m\ge 2$.
Note that $L^{-2}(H)$ is isomorphic to one member of $\{K_3, K_{1,3}, J_1, J_2\}$.
If $L^{-2}(H)\cong K_3$ or $J_2$, then $L^{m-2}(G)\cong G\cong K_3$ or $J_2$, respectively, as $d(x)=d(y)=2$. It contradicts $G\in \G$.
Now, $L^{-2}(H)$ is isomorphic to one member of $\{K_{1,3}, J_1\}$.
Since $d(x)=d(y)=2$ as well as line graphs are claw-free and contain no parallel edges, it shows that $m=2$.
As $d(x)=d(y)=2$, $\{x,y\}\subseteq F$ and there is a common end vertex of edges $x$ and $y$ of degree three, which shows $m\ge 3$ by (\ref{e_m}). It contradicts the fact we got before that $m=2$.
\end{proof}

\subsection{The $k$-Triangular Index}

Before establishing the bounds for $t_k(G)$, we need some lemmas.

\begin{theorem}[Niepel, Knor and \v{S}olt\'{e}s, Lemma 1(1) of \cite{NiKS96}]\label{lowerbound}
Let $G$ be a simple graph with $\delta(G)\geq 3$. Then, $\delta(L^i(G))\geq 2^i(\delta(G)-2)+2$ for each integer $i\geq 0$.
\end{theorem}

By the definition of line graphs, if $G$ is a regular graph, then for each integer $i\geq 0$, we always have
$\delta(L^i(G)) = 2^i(\delta(G)-2)+2$, and so the lower bound in Theorem \ref{lowerbound}
is best possible in this sense.

\begin{lemma} \label{delta-2}
Let $G \in \G$ be a simple graph with $\delta=\delta(G)$. Each of the following holds for each integer $i > 0$.\\
(i) If $\delta\ge 3$, then $L^i(G)$ is $(2^{i-1}(\delta-2))$-triangular.\\
(ii) If $\delta\le 2$, then $L^{\widetilde d+i}(G)$ is $(2^{i-1}(\delta_0-2))$-triangular where
$\delta_0=\delta(L^{\widetilde d(G)}(G))$. In particular, $L^{\widetilde d+i}(G)$ is $2^{i-1}$-triangular.
\end{lemma}

\begin{proof}
Let $e_1e_2\in E(L(G))$ be an arbitrary edge in $L(G)$. Then there exists a vertex $u\in V(G)$
such that $\{e_1,e_2\}\subset E_G(u)$.
Suppose $\delta\ge 3$. In general, as $L(G)[E_G(u)]\cong K_{d(u)}$, the edge $e_1e_2$ lies in at least $d(u)-2\geq \delta-2\ge 1$ distinct triangles.
It means that $L(G)$ is $(\delta-2)$-triangular.
By Theorem \ref{lowerbound}, for each integer $i>0$, $\delta(L^{i-1}(G))\geq 2^{i-1}(\delta-2)+2\geq 3$.
It follows that $L^i(G)$ is $(2^{i-1}(\delta-2))$-triangular and (i) is proved.

To show (ii), as $\delta_0\geq 3$, it follows by (i)
that $L^{\widetilde d+i}(G)=L^i(L^{\widetilde d}(G))$ is $(2^{i-1}(\delta_0-2))$-triangular.
\end{proof}

\begin{theorem}\label{kindex}
Let $k\ge 2$ be an integer and $G\in \G$ be a simple graph with $\delta = \delta(G)$ and $\widetilde d=\widetilde d(G)$. Each of the following holds.
\\
(i) Being $k$-triangular is line graph stable.
\\
(ii)
\begin{equation} \label{tk}
t_k(G)\le\left\{
\begin{aligned}
& \widetilde d+1+\lceil \lg k\rceil, &&\mbox{ if }\delta\le 2;\\
& 1+\left\lceil\lg\frac{k}{\delta-2}\right\rceil, && \mbox{ if } 3\le\delta\le k+1;\\
& 1, &&{\rm otherwise}.
\end{aligned}
\right.
\end{equation}
Moreover, the equality holds for sufficiently large $k$ when $\delta\le k+1$.
\end{theorem}

\begin{proof}
(i) Suppose $G\in\G$ is a simple $k$-triangular graph for given $k\ge 2$. Then $\delta(G)\ge k+1 \ge 3$.
Pick an edge $e_1e_2\in E(L(G))$. To show that $L(G)\in \T_k$, it is enough to prove that $e_1e_2$ lies in at least $k$ distinct triangles in $L(G)$.
Let $x$ be the common vertex of $e_1$ and $e_2$ in $G$,
and $X$ be the set of all edges adjacent with both edges $e_1$ and $e_2$, that is, $X=E_G(e_1)\cap E_G(e_2)$.
If $d(x)\ge k+2$, then $|X|\ge k$.
It means that $e_1e_2$ lies in at least $k$ distinct triangles in $L(G)$.
Now, we consider that $d(x)=k+1$. Since $G\in \T_k$ is a simple graph, $G[N_G(x)]$ is a complete graph and then $e_1e_2$ lies in at least $k$ distinct triangles in $L(G)$.

(ii) Let $t=t_k(G)$. First, we consider the situation when $\delta\le 2$.
As $k\ge 2$, by the definition of $\widetilde d$, we have $t\geq \widetilde d$.
If $t<\widetilde d+2$, then $t< \widetilde d+1+\lceil\lg k\rceil$ as $k\geq 2$.
Assume next that $k$ is so large that $t\geq \widetilde d+2$. As $L^{t}(G)\in\T_k$ but
$L^{t-1}(G)\notin \T_k$, by Lemma \ref{delta-2}(ii),
$2^{t-\widetilde d-2}<k\leq 2^{t-\widetilde d-1}$. Then algebraic manipulation leads to
$t-\widetilde d-2<\lg k\leq t-\widetilde d-1$, which means that $\lceil\lg k\rceil=t-\widetilde d-1$.
Hence we conclude that $t= \widetilde d+1+\lceil\lg k\rceil$.

Now, we suppose that $\delta\ge 3$. 
If $\delta\ge k+2$, then $L(G)\in\T_{\delta-2}$ by Lemma \ref{delta-2}(i), which implies that $L(G)\in\T_{k}$ and then $t\le 1$.

If $\delta\le k+1$ and $t\ge 2$, then, by Lemma \ref{delta-2}(i), for each integer $i > 0$,
$L^i(G)$ is $(2^{i-1}(\delta-2))$-triangular. So $2^{t-2}(\delta-2)< k \leq 2^{t-1}(\delta-2)$ by the definition of $t=t_k(G)$.
It follows that $t=1+\left\lceil\lg{\frac{k}{\delta-2}}\right\rceil$. Then, $t\le 1+\left\lceil\lg{\frac{k}{\delta-2}}\right\rceil$ when $3\le\delta\le k+1$.
\end{proof}

\section{Proof of Theorem \ref{mainthm}}

An \emph{elementary subdivision} of a graph $G$ at an edge $e=uv$ is a graph
$G(e)$ obtained from $G-e$ by adding a new vertex $v_e$ and two new edges $uv_e$ and $v_ev$.
For a subset $X\subseteq E(G)$, we define $G(X)$ to be the graph obtained from $G$ by elementarily
subdividing every edge of $X$.

\begin{lemma}\label{kl}
For an integer $k>1$, if $G\in\G$ is a $k$-triangular simple graph and
$X\subset E(G)$ with $|X|=s$ where $1\leq s< k$, then $G-X\in\T_{k-s}$.
\end{lemma}

\begin{proof}
Pick $e\in E(G-X)$. Since $G\in\T_k$, edge $e$ lies in at least $k$ distinct triangles in $G$,
say $C_1^e,C_2^e,\ldots, C_k^e$. As $E(C^e_i\cap C^e_j)=\{e\}$ for each $\{i,j\}\subseteq [1,k]$ and $|X|=s<k$,
there exist $k-s$ such triangles $C^e_{i'}$ where $i'\in [1,k]$ such that $E(C^e_{i'})\cap X=\emptyset$.
It follows that $G-X\in\T_{k-s}$.
\end{proof}

\begin{lemma} \label{st}
Given two non-negative integers $s$ and $t$. If $G\in\G$ is
a $(s+t+1)$-triangular simple graph, then $G$ is $(s,t)$-supereulerian.
\end{lemma}

\begin{proof}
For any $X,Y\subset E(G)$ with $X\cap Y=\emptyset$, $|X|=s_1\leq s$ and $|Y|\leq t$. Then $|X\cup Y|\leq s+t$.
Let $H=G-(X\cup Y)$.
By Lemma \ref{kl}, $H\in\T_1$. It follows that $H$ is collapsible by Theorem \ref{collapsible}(i).
Let $X=\{x_1,x_2,\ldots, x_{s_1}\}$. Then $V(G(X))=V(G)\cup \{v_{x_1},v_{x_2},\ldots, v_{x_{s_1}}\}$. Note that
$G(X)-Y-\{v_{x_1},v_{x_2},\ldots, v_{x_{s_1}}\}=H$ is collapsible. Since every edge of $(G(X)-Y)/H$ lies in a cycle of length 2,
which implies that $(G(X)-Y)/H$ is collapsible by Theorem \ref{collapsible}(i). It indicates that $G(X)-Y$ is collapsible by
Theorem \ref{collapsible}(ii) as $H$ is collapsible. Then $G(X)-Y$ is supereulerian,
which means that $G(X)-Y$ has a spanning eulerian subgraph $J$.
Note that $d_{G(X)-Y}(v_{x_i})=2$ for each $i\in[1,s_1]$. Then subgraph $J$ contains
all edges incident with some $v_{x_i}$, which means that $G-Y$ has a spanning
eulerian subgraph $J'$ containing $X$, and so $G$ is $(s,t)$-supereulerian.
\end{proof}

\begin{proof}[of Theorem \ref{mainthm}]
Combine Theorem \ref{old1}(ii), Theorem \ref{kindex}(ii) and Lemma \ref{st}, and then we complete the proof of it.
\end{proof}

\acknowledgements
The authors would like to thank the two anonymous referees for their comments, which helped us to improve the manuscript.


\end{document}